\documentclass[11pt]{article}
\usepackage{amsthm, amsmath, amssymb, amsfonts, url, booktabs, tikz, setspace, fancyhdr, bm}
\usepackage{subcaption}
\usepackage[hidelinks]{hyperref}
\usepackage{geometry}
\usepackage{hyperref, enumerate}
\usepackage[shortlabels]{enumitem}
\usepackage[babel]{microtype}
\usepackage[english]{babel}
\usepackage[capitalise]{cleveref}
\usepackage{comment}
\usepackage{bbm}
\usepackage{csquotes}
\usepackage{mathabx}

\counterwithin{figure}{section}

\newtheorem{theorem}{Theorem}[section]
\newtheorem{prop}[theorem]{Proposition}

\newtheorem{conj}[theorem]{Conjecture}
\newtheorem{claim}[theorem]{Claim}

\newtheorem{problem}[theorem]{Problem}

\theoremstyle{definition}

\newtheorem{ques}[theorem]{Question}
\newtheorem{strategy}{Strategy}

\theoremstyle{remark}

\newlist{Case}{enumerate}{3}
\setlist[Case, 1]{%
    label           =   {\bfseries Case \arabic*.},
    labelindent=1em ,labelwidth=1cm, labelsep*=1em, leftmargin =!
}
\setlist[Case, 2]{%
    label           =   {\bfseries Case \arabic{Casei}.\arabic*.},
    labelindent=-2em ,labelwidth=1cm, labelsep*=1em, leftmargin =!
}
\setlist[Case, 3]{%
    label           =   {\bfseries Case \arabic{Casei}.\arabic{Caseii}.\arabic*.},
    labelindent=-2em ,labelwidth=1cm, labelsep*=1em, leftmargin =!   
}

\newenvironment{poc}{\begin{proof}[Proof of claim]}{\end{proof}}

\usepackage{todonotes}

\title{Strong Ramsey game on two boards }
\author{
Jiangdong Ai\thanks{School of Mathematical Sciences and LPMC, Nankai University, Tianjin, P.R. China. Emails: {\texttt jd@nankai.edu.cn, xinyan@mail.nankai.edu.cn}. Supported by the National Natural Science Foundation of China under grant No.12161141006 and No.12401456.}
\and
Jun Gao\thanks{Extremal Combinatorics and Probability Group (ECOPRO), Institute for Basic Science (IBS), Daejeon, South Korea. Emails: {\texttt \{jungao, zixiangxu\}@ibs.re.kr}. Supported by the Institute for Basic Science (IBS-R029-C4).}
\and
Zixiang Xu\footnotemark[2]
\and
Xin Yan\footnotemark[1] 
}

\begin{document}

\maketitle

\begin{abstract}
The strong Ramsey game $R(\mathcal{B}, H)$ is a two-player game played on a graph $\mathcal{B}$, referred to as the board, with a target graph $H$. In this game, two players, $P_1$ and $P_2$, alternately claim unclaimed edges of $\mathcal{B}$, starting with $P_1$. The goal is to claim a subgraph isomorphic to $H$, with the first player achieving this declared the winner. A fundamental open question, persisting for over three decades, asks whether there exists a graph $H$ such that in the game $R(K_n, H)$, $P_1$ does not have a winning strategy in a bounded number of moves as $n \to \infty$.

In this paper, we shift the focus to the variant $R(K_n \sqcup K_n, H)$, introduced by David, Hartarsky, and Tiba, where the board $K_n \sqcup K_n$ consists of two disjoint copies of $K_n$. {We prove that there exist infinitely many graphs $H$ such that $P_1$ cannot win in $R(K_n \sqcup K_n, H)$ within a bounded number of moves through a concise proof.} This perhaps provides evidence for the existence of examples of the longstanding open problem mentioned above.

\end{abstract}

\section{Introduction}
 The study of positional games began with the pioneering work of Hales and Jewett~\cite{1963HalesJ}, followed by notable contributions from Erd\H{o}s and Selfridge~\cite{1973Erdos}. Beck's prolific research~\cite{1981Beck,1985Beck,1996Beck,2002DMBeck,1982Beck2}, starting in the early 1980s, significantly defined and expanded this topic. Among the many variants~\cite{2009OnlineRamsey,2020COnlonRamseyGame,2015RSA,2023Guo,2014Book,2011JAMS}, strong positional games stand out as the most intuitive and natural.

 The strong Ramsey game $R(\mathcal{B},H)$ is played on a finite or infinite graph $\mathcal{B}$, called the board, with target graph $H$. Two players $P_{1}$ and $P_{2}$, take turns selecting unclaimed edges of $\mathcal{B}$, with $P_1$ making the initial move. The first player to successfully establish an isomorphic copy of $H$ wins and if neither player achieves this within a finite time, the game is declared a draw.

 In the context of fixed graphs, a general strategy-stealing argument establishes that $P_{2}$ cannot have a winning strategy. Moreover, the classical Ramsey theorem~\cite{1930Ramsey} states that for any given graph $H$, there exists a natural number $n_{0}(H)$ such that for any integer $n\ge n_{0}(H)$, any edge-colored clique $K_{n}$ with two colors contains a monochromatic copy of $H$. Consequently, $P_{1}$ is guaranteed to have a winning strategy for sufficiently large $n$. However, the strategy derived from this argument is non-explicit, and very few explicit strategies are known.

While it is known that as $n$ becomes sufficiently large, $P_{1}$ is assured of having a winning strategy, a pertinent query arises: Can $P_{1}$ ensure victory within a fixed number of moves, denoted by $c_{H}$, which solely relies on the target graph $H$? To formally address this, we introduce the extremal function $L(\mathcal{B},H)$, which represents the minimum number of moves required for $P_{1}$ to secure victory, regardless of $P_{2}$'s level of intelligence.

\begin{ques}
    For given graph $H$ and $\mathcal{B}$, determine the extremal function $L(\mathcal{B},H)$ as $n\rightarrow\infty$.
\end{ques}

By the definition of the Tur\'{a}n number of graph $H$,  it follows that for any given graph $H$, we have $L(K_{n},H)\le \textup{ex}(n,H)$, where $\textup{ex}(n,H)$ is the maximum number of edges in an $n$-vertex $H$-free graph.

A major conjecture in this area, proposed by Beck~\cite{2002DMBeck}, is stated as follows. Remarkably, Beck includes this problem among his ``7 most humiliating open problems'' and deems even the case \( t = 5 \) to be seemingly intractable.

\begin{conj}
    Let $n$ be a sufficiently large number, for any $t\ge 3$, 
    \begin{equation}
        L(K_{n},K_{t})\le c_{t},
    \end{equation}
    where $c_{t}$ is a constant which depends only on $t$.
\end{conj}

It was proven in~\cite{2002DMBeck} that the above conjecture holds for $t=3,4$.\footnote{For $t=4$, Bowler and Gut showed that  
the proof was inaccurate and fixed it~\cite{bowler2023k4game}.} However, surprisingly when $t\ge 5$, the best known upper bound is $(1-\frac{1}{t-1}+o(1))\binom{n}{2}$, which is trivial by the so-called Tur\'{a}n Theorem~\cite{1941Turan}.

Hefetz, Kusch, Narins, Pokrovskiy, Requil\v{e}, and Sarid~\cite{hefetz2017strong} explored a natural generalization from graphs to hypergraphs, completely altering the intuition about this phenomenon. Specifically, they proved the existence of a \(5\)-uniform hypergraph \(\mathcal{H}\) such that, in the game \(R(K_{n}^{(5)}, \mathcal{H})\), \(P_1\) cannot secure a win within a bounded number of moves. Subsequently, in~\cite{david2020strong}, David, Hartarsky, and Tiba investigated another intriguing generalization where the board consists of two vertex-disjoint copies of \(K_{n}\). They demonstrated that when \(H = K_{6} \setminus K_{4}\), \(P_1\) cannot win the game \(R(K_{n} \sqcup K_{n}, H)\) within a bounded number of moves.

For given graphs \( H \) and $\mathcal{B}$, to prove that \( C \) is an upper bound for \( L(\mathcal{B}, H) \), we must show that, regardless of how strategically \( P_2 \) plays, \( P_1 \) can always secure a win within \( C \) moves. Conversely, to establish that \( D \) is a lower bound for \( L(\mathcal{B}, H) \), we need to demonstrate that, no matter what strategy \( P_1 \) adopts, \( P_2 \) can always ensure the game lasts at least \( D \) moves by avoiding defeat within the first \( D - 1 \) moves. Moreover, one can see for any fixed graph $H$, if $P_{1}$ can win the game $R(K_{n}\sqcup K_{n},H)$ within a bounded number of moves, then $P_{1}$ can also win the game $R(K_{n},H)$.

The main result of this paper is to provide infinitely many examples $H$ such that $P_{1}$ cannot win the game $R(K_{n}\sqcup K_{n},H)$ within a bounded number of moves. More precisely, for integers $t\ge 2$ and $s\ge 0$, let $K_{2,t}(s)$ be the complete bipartite graph plus $s$ edge which is incident to some vertex in the smaller part, that is, $V(K_{2,t}(s)):=\{a_{1},a_{2},b_{1},b_{2},\ldots,b_{t},c_{1},\dots, c_{s}\}$, $a_{i}b_{j}\in E(K_{2,t}(s))$ for any $i\in\{1,2\}$ and $j\in\{1,2,\ldots,t\}$, and $a_{1}c_{i}\in E(K_{2,t}(s))$ for any $i\in [s]$. 

\begin{theorem}\label{thm:P1Lose}
For any $t\ge 3$ and constant $C$, there exists some constant $n_{0}$ such that for any $n\ge n_{0}$, we have
\begin{equation*}
    L(K_{n}\sqcup K_{n},K_{2,t+1}(t-2))> C.
\end{equation*}
\end{theorem}

In particular, the condition $t\ge 3$ is optimal. By definition, when $t=2$, $K_{2,t+1}(t-2)$ is exactly the complete bipartite graph $K_{2,3}$. We can show that $P_{1}$ can easily win the game $R(K_{n}\sqcup K_{n},H)$ when $H$ is $K_{2,3}$. 
Moreover, we can also show that $P_{1}$ can win the game $R(K_{n},C_{\ell})$, where $C_{\ell}$ is the cycle of length $\ell\ge 3$.

\begin{theorem}\label{thm:K23}
Let $n$ be a large enough number, then $P_{1}$ can win the following games using a constant number of moves.
\begin{enumerate}
    \item $R(K_{n}\sqcup K_{n},K_{2,3})$;
    \item $R(K_{n},C_{\ell})$ for any $\ell\ge 3$.
\end{enumerate}
\end{theorem}

The main advantage is, our proofs are concise and straightforward, without relying on computer-assisted exhaustive searches. For convenience, we will use \( K_{\infty} \) to denote the complete graph \( K_n \) when \( n \) is sufficiently large.

\section{$P_{1}$ cannot win in $R(K_{\infty}\sqcup K_{\infty},K_{2,t+1}(t-2))$}
    Let $\mathcal{B}_1 \cong K_{\infty}$, and  $\mathcal{B}_2 \cong K_{\infty}$ be two disjoint boards.
    Without loss of generality, we can assume $P_{1}$ claims the first edge in $\mathcal{B}_1$.
    For a given positive integer $i$, let $G_{i}^{(1)}$ and $G_{i}^{(2)}$ be the graphs in $\mathcal{B}_{1}$ and $\mathcal{B}_{2}$ that $P_{1}$ claimed after round $i$, respectively, and let $F_{i}^{(1)}$ and $F_{i}^{(2)}$ be the graphs in $\mathcal{B}_{1}$ and $\mathcal{B}_{2}$ that $P_{2}$ claimed after round $i$, respectively.

    In order to show the lower bound on $L(K_{\infty}\sqcup K_{\infty}, K_{2,t+1}(t-2))$, we next propose a strategy for $P_{2}$ and show that, if $P_{2}$ strictly obeys this strategy, then $P_2$ can successfully prevent $P_{1}$ from winning the game within $C$ moves for any constant $C$. 

    Note that $e(K_{2,t+1}(t-2)) =3t$.
    Let $\mathcal{F}$ be the family consisting of all graphs obtained from $K_{2,t+1}(t-2)$ by removing $t-1$ edges.
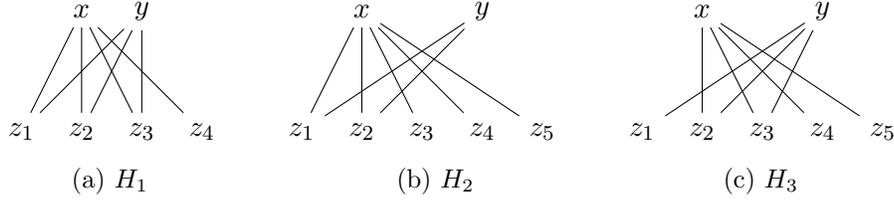
\begin{figure}[htbp]
    \centering

    \begin{subfigure}[b]{0.25\textwidth} 
    \centering
    \begin{tikzpicture}[scale=0.4]
        \node (A1) at (2,0) {$x$};
        \node (A2) at (4,0) {$y$};
        \node (B1) at (0,-4) {$z_1$};
        \node (B2) at (2,-4) {$z_2$};
        \node (B3) at (4,-4) {$z_3$};
        \node (B4) at (6,-4) {$z_4$};
        \foreach \from in {A1,A2}
        \foreach \to in {B1,B2,B3}
        \draw (\from) -- (\to);
        \draw (A1) -- (B4);
    \end{tikzpicture}
    \caption{$H_1$}
    \end{subfigure}
    \begin{subfigure}[b]{0.25\textwidth} 
    \begin{tikzpicture}[scale=0.4]
        \node (A1) at (2,0) {$x$};
        \node (A2) at (6,0) {$y$};
        \node (B1) at (0,-4) {$z_1$};
        \node (B2) at (2,-4) {$z_2$};
        \node (B3) at (4,-4) {$z_3$};
        \node (B4) at (6,-4) {$z_4$};
        \node (B5) at (8,-4) {$z_5$};
        \foreach \from in {A1}
        \foreach \to in {B1,B2,B3,B4,B5}
        \draw (\from) -- (\to);
        \draw (A2) -- (B1);
        \draw (A2) -- (B2);
    \end{tikzpicture}
    \caption{$H_2$}
    \end{subfigure}
    \begin{subfigure}[b]{0.25\textwidth} 
    \centering
    \begin{tikzpicture}[scale=0.4]
        \node (A1) at (2,0) {$x$};
        \node (A2) at (6,0) {$y$};
        \node (B1) at (0,-4) {$z_1$};
        \node (B2) at (2,-4) {$z_2$};
        \node (B3) at (4,-4) {$z_3$};
        \node (B4) at (6,-4) {$z_4$};
        \node (B5) at (8,-4) {$z_5$};
        \foreach \from in {A1}
        \foreach \to in {B2,B3,B4,B5}
        \draw (\from) -- (\to);
        \draw (A2) -- (B1);
        \draw (A2) -- (B2);
        \draw (A2) -- (B3);
    \end{tikzpicture}
    \caption{$H_3$}
    \end{subfigure}
    \caption{For $t=3$, all possible graphs obtained from $K_{2,4}(1)$ by removing $2$ edges (under isomorphism)
    .}
    \label{fig:enter-label}
\end{figure}

Given a graph $G$, we say a vertex $u$ is a core of $G$ if $d_{G}(u)\ge 3$.
We say a vertex is \emph{unused} if, at that moment, there is no edge already being claimed on the board that is incident to this vertex. 
The term \emph{check} refers to a situation where a player will win in the next move if the opponent gives up the next move.
We say the player $P$ \emph{is 
in check} where  the opponent will win the game if $P$ gives up the next move.
\begin{figure}[htbp]
    \centering
        \begin{tikzpicture}[scale=0.4]
        \node (A1) at (2,0) {$u$};
        \node (A2) at (6,0) {$v$};
        \node (B1) at (0,-4) {$x_1$};
        \node (B2) at (2,-4) {$x_2$};
        \node (B3) at (4,-4) {$x_3$};
        \node (B4) at (6,-4) {$x_4$};
        \node (B5) at (8,-4) {$x_5$};
        \foreach \from in {A1}
        \foreach \to in {B1,B2,B3,B4,B5}
        \draw (\from) -- (\to);
        \draw (A2) -- (B1);
        \draw (A2) -- (B2);
        \draw (A2) -- (B3);
    \end{tikzpicture}
    \caption{The graph $K_{2,3}(2)$: key of $P_{2}$ for $t=3$}
    \label{fig:KeyforP2}
\end{figure}
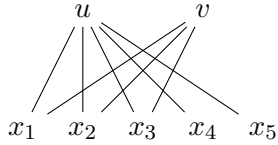

\begin{strategy}[Strategy for $P_{2}$]
Suppose that the game is not finished after $P_{1}$'s $i$-th move, $P_{2}$'s $i$-th move should obey the following rules. Also note that these strategies have different priorities, and $P_{2}$ will only execute the next strategy if the previous one does not occur.

\begin{Case}
    \item If $P_{2}$ has a chance to win the game after the $i$-th move, then $P_{2}$ claims that edge and wins the game.

    \item If Case~1 does not occur, and $G^{(1)}_{2t+1}\cong F$ for some $F \in \mathcal{F}$. 
    Recall that $F$ can be obtained from $K_{2,t+1}(t-2)$ by removing $t-1$ edges. Let $x,y$ be the core of the original $K_{2,t+1}(t-2)$, $A$ be the common neighbor of $x,y$ in $F$, and let $B$ be the set of the vertices that are adjacent to exactly one of $x,y$ in $F$.
    Clearly, we have $e(F) =2|A|+ |B| =2t+1$, which implies that $|B|$ is odd. Let $|A|=a, |B| =2b-1$, where $b\ge 1$.
    Since $2t-1\ge |A|+|B| =2t+1-|A|$, we have $a\ge2$.
    \begin{enumerate}
    \item In round $2t+1$, $P_2$ picks an arbitrary vertex  $w\in B$ and claims the unused edge $xw$ or $yw$. Then we set $B = B\setminus \{w\}$.

    \item Suppose that in some following rounds, $P_{1}$ claims the edge $xw$ or $yw$ for some unused vertex $w$, then $P_{2}$ claims $yw$ or $xw$, respectively.

    \item  Suppose that in some following rounds, $P_1$ claims the edge $xw$ or $yw$ for some vertex $w \in B$, then $P_2$ picks another vertex $w'\in B$ and claims the unused edge $xw$ or $yw$. Then we set $B = B\setminus \{w,w'\}$.
    \end{enumerate}      
    \item If neither Case~1 nor Case~2 occurs, then $P_{2}$ should try to build a copy of $K_{2,t}(t-1)$, which is shown in Figure~\ref{fig:KeyforP2} for $t=3$, in $\mathcal{B}_{2}$. More precisely, to build the graph $K_{2,t}(t-1)$, $P_{2}$ should obey the following order if, during this process, none of the preceding cases occur.

    \begin{enumerate}
        \item $P_{2}$ first claims an arbitrary edge $ux_{1}$, where $u$ and $x_{1}$ are not incident to any claimed edge. Then $P_{2}$ can repeatedly pick $2t-2$ more unused vertices, say $x_{2},x_{3},\dots ,x_{2t-1}$ and claim $ux_{2},ux_{3},\dots,ux_{2t-1}$ in this order.
        \item $P_{2}$ next picks an unused vertex, say $v$, and then claims $t$ of the $2t-1$ edges $vx_{1}$,$vx_{2}$, $\ldots,vx_{2t-1}$ with a preference for edges $vx_{j}$ according to the following rules:
        if at that moment $P_1$ claims a star then  we give priority to vertex $x_j$ with degree one in the graph claimed by $P_1$ in $\mathcal{B}_2$, otherwise  we prioritize vertex $x_j$ with the highest degree in the graph claimed by $P_1$ in $\mathcal{B}_2$.
    \end{enumerate} 
    \item If $P_{2}$ has claimed a copy of $K_{2,t}(t-1)$ in $\mathcal{B}_{2}$ with vertex set $\{u,v,x_{1},x_{2},\dots,x_{2t-1}\}$, and none of the events in Case~1 and Case~2 occurs, then $P_{2}$ picks an unused vertex $w$. If $P_1$ does not have a $K_{2,t}$ with $u$ as a core, then $P_2$ claims the edge $vw$, otherwise $P_2$ claims $uw$.
\end{Case}
\end{strategy}

Next, we would like to show that, if $P_{2}$ strictly obeys the above strategy, then $P_{2}$ can guarantee that $P_{1}$ cannot win the game within a finite number of moves.

We need a series of claims as follows.

\begin{claim}
   If $P_{2}$ wants to claim a copy of graph $K_{2,t}(t-1)$ in $\mathcal{B}_{2}$, $P_2$ can always achieve it after $3t-1$ moves, provided that during the process none of Case~1 and Case~2 occur.
\end{claim}
\begin{poc}
    By the embedding rules listed in Case~3 above, $P_{2}$ can always find a copy of $K_{1,2t-1}$ with edges $ux_{1},ux_{2},\dots,ux_{2t-1}$, and as $v$ is an unused vertex, by the pigeonhole principle, $P_{2}$ can guarantee to claim $t$ of edges $vx_{1},vx_{2},\dots,vx_{2t-1}$.
\end{poc}

Let $\Delta(i)=e(G_{i}^{(1)})-e(F_{i}^{(1)})$. Note that if $P_1$ wins in round $i$, then $\Delta(i)$ does not exist.

\begin{claim}\label{claim:delta>8}
    If $P_{1}$ finally claims a copy of $K_{2,t+1}(t-2)$ in $\mathcal{B}_{1}$, then there must be some $i$ such that $\Delta(i)=3t$.
\end{claim}
\begin{poc}
If $G^{(1)}_{2t+1}\cong F$ for some $F \in \mathcal{F}$, then $P_{2}$ should start to defend in round $2t+1$, which implies $\Delta(2t+1)=2t$.
Recalling the definition of the graph
$F$, we know that $a+b = t+1$. 
So if $P_{2}$ obeys the strategy, then $P_{1}$ cannot claim a copy of $K_{2,t+1}(t-2)$ with core vertices $x,y$ within a finite number of moves. That means, if $P_1$ finally claims a copy of $K_{2,t+1}(t-2)$ in $\mathcal{B}_1$, then $P_{1}$ should find another pair of core vertices $x',y'$, where $|\{x,y\}\cap\{x',y'\}|\le 1$.
If $|\{x,y\}\cap\{x',y'\}|= 1$, consider the codegree of the core vertices $x',y'$, $P_{1}$ needs to claim at least $t+1$ more edges in $\mathcal{B}_{1}$ which are not equal to $xw$ or $yw$ for some $w\notin \{x,y\}$.
If $|\{x,y\}\cap\{x',y'\}|= 0$, consider the degree of the core vertices $x',y'$, $P_{1}$ needs to claim at least $3t-4$ more edges in $\mathcal{B}_{1}$ which are not equal to $xw$ or $yw$ for some $w\notin \{x,y\}$.
Since $t\ge 3$, we have $3t-4\ge t+1$.
Based on $P_{2}$'s strategy, there must be some integer $i$ such that $\Delta(i)\ge 2t+t=3t$.

If $G^{(1)}_{2t+1}\ncong F$ for any $F \in \mathcal{F}$, we have $G^{(1)}_{3t}$ does not contain $K_{2,t+1}(t-2)$. Note that $P_2$ does not claim any edges in $\mathcal{B}_{1}$, thus $\Delta(3t)=3t$.
\end{poc}

\begin{claim}\label{claim:P1CannotWin1}
     $P_{1}$ can not claim a copy of $K_{2,t+1}(t-2)$ in $\mathcal{B}_{1}$.
\end{claim}
\begin{poc}
Suppose on the contrary that $P_{1}$ finally claims a copy of $K_{2,t+1}(t-2)$ in $\mathcal{B}_{1}$. By Claim~\ref{claim:delta>8}, there exists some $i$ such that $e(F_{i}^{(2)})=3t-1$.  
By $P_{2}$'s strategy, $F_{i}^{(2)}\cong K_{2,t}(t-1)$, without loss of generality, we further assume that $E(F_{i}^{(2)})=\{ux_{s},vx_{k},s\in [2t-1],k\in [t]\}$. Note that $e(G^{(2)}_{i}) \le e(F_{i}^{(2)})-1 =3t-2$, and $\Delta(j) \le 3t-1$ for all $j\le i$. By Claim~\ref{claim:delta>8}, $P_1$ can not win in round $i+1$.
This means that $G^{(2)}_{i+1}$ must contain $vx_{j}$ for all $j\in [2t-1] \setminus [t]$, otherwise $P_2$ will win in round $i+1$.
This implies that $$\Delta(i+1) \le(3t-1) -(t-1)+1= 2t+1.$$ 
By Claim~\ref{claim:delta>8}, there exists some $m$ such that $\Delta(m)=3t$. This means there exists a rounds $m'$ after round $i$ such that $P_2$ claims an edge in $\mathcal{B}_2$ and $P_1$ claims an edge in $\mathcal{B}_1$ in round $m'+1$ such that the game is not over,
which implies that $P_2$ claimed a copy of $K_{2,t+1}(t-2)$ in round $m'+1$, a contradiction.
\end{poc}

\begin{claim}\label{claim:P1CannotWin2}
     $P_{1}$ can not claim a copy of $K_{2,t+1}(t-2)$ in $\mathcal{B}_{2}$.   
\end{claim}
\begin{poc}
Suppose on the contrary that $P_{1}$ finally claims a copy of $K_{2,t+1}(t-2)$ in $\mathcal{B}_{2}$. According to the $P_{2}$'s strategy, we have $e(G_{i}^{(2)}) \le e(F_{i}^{(2)})-1$, which implies that $e(G_{i+1}^{(2)}) \le e(F_{i}^{(2)})$. Therefore, there exists some $i$ such that $e(F_{i}^{(2)})=3t-1$ and $F_{i}^{(2)}\cong K_{2,t}(t-1)$. 
Similarly, we assume that $E(F_{i}^{(2)})=\{ux_{s},vx_{k},s\in [2t-1],k\in [t]\}$ and  $G^{(2)}_{i+1}$ must contain $vx_{s}$ for all $s\in [2t-1] \setminus [t] $.
Since $e(K_{2,t+1}(t-2)) =3t$, there also exists some $j$ such that $e(F_{j}^{(2)})=3t$. 

We claim that $P_2$ can not be in check in round $j$.
Suppose that $P_2$ is in check in round $j$, then $G^{(2)}_j$ is isomorphic to the graph obtained by deleting an edge from $K_{2,t+1}(t-2)$. Note that, $e(G_j^{(2)}) = 3t-1$ and $v$ has degree $t-1$ or $t$ in $G_j^{(2)}$.
If $v$ is a core of $G^{(2)}_j$, then $G^{(2)}_j \cong K_{2,t}(t-1)$ as $d(v) \le t$. Let $u'$ be another core of $G^{(2)}_j$.
Since the vertex $v$ remains unused until $P_2$ picks it,
we conclude that the first $2t-1$ moves of $P_1$ on $\mathcal{B}_2$ form a star centered at $u'$.
Since $P_2$ claims $t$ of the $2t-1$ edges $vx_1,vx_2,\dots,vx_{2t-1}$ with a preference for edges $vx_j$, where $x_j$ has degree one in the graph claimed by $P_1$ in $\mathcal{B}_2$, $t-1$ of the edges $u'x_1,u'x_2,\dots u'x_t$ will belong to $G_{j+1}^{(2)}$.
Since the edges $vx_1,vx_2,\dots vx_t$ have been claimed by $P_2$, we derive that $P_2$ is not in check in round $j$.
If $v$ is not a core of $G^{(2)}_j$, then the degree of $v$ in $G^{(2)}_j$ is 2 and $t=3$, which implies that $x_4$ and $x_5$ are the two cores of  $G^{(2)}_j$. Since at the moment $P_2$ picks $v$, $P_1$ does not claim a star in $\mathcal{B}_2$, by the strategy of $P_2$ for claiming the $3$ edges from $vx_1,vx_2,\dots,vx_5$, one of $vx_4$ and $vx_5$ will be claimed by $P_2$, which leads to a contradiction.

Since \( e(G^{(2)}_j) \leq 3t-1 \) and \( u, v \) do not belong to the same \( K_{2,t} \) in \( G^{(2)}_j \), we deduce that at least one of the vertices \( u \) or \( v \) does not belong to any \( K_{2,t} \) in \( G^{(2)}_j \). Without loss of generality, assume that \( v \) does not belong to any \( K_{2,t} \) in \( G^{(2)}_j \). in round \( j+1 \), \( P_2 \) claims the edge \( uw \) for some unused vertex \( w \). Since \( P_2 \) is not in check and $P_1$ is in check, \( P_1 \) must respond by claiming \( vw \) in the next round. In each subsequent round, \( P_2 \) will continue to claim an edge \( uw' \) for some unused vertex \( w' \), forcing \( P_1 \) to claim \( vw' \) in response. This process ensures that \( P_1 \) cannot claim a copy of $K_{2,t+1}(t-2)$, leading to a contradiction.
\end{poc}

By~\cref{claim:P1CannotWin1} and~\cref{claim:P1CannotWin2}, we can see that $P_1$ cannot win within a finite number of moves. This finishes the proof.

\section{$P_{1}$ always wins for various games: Proof of~\cref{thm:K23}}
In this section, we prove the results in~\cref{thm:K23} separately.

\subsection{The game $R(K_{n},C_{\ell})$}
We first demonstrate that \( P_1 \) can claim a path of length \( \ell - 2 \) in such a way that \( P_2 \) is unable to form a path of length \( \ell - 2 \) with the same endpoints after \( \ell - 2 \) moves. To see this,
after $\ell -3$ moves, $P_1$ can claim a path of length $\ell-3$, denoted by $v_{0}v_{1}\cdots v_{\ell-3}$.
If $P_2$ does not claim a path of length $\ell-3$ after $\ell-3$ moves, then $P_{1}$ can pick a new vertex $v$, which is not adjacent to any edge claimed by $P_{2}$. Regardless of $P_{2}$'s moves on the $(\ell-2)$-th round, $P_{1}$ can claim a path $v_{0}v_{1}\cdots v_{\ell-3}v$, but $P_{2}$ cannot claim a path of length $\ell-2$ with endpoints $v_{0}$ and $v$. Otherwise, suppose that $P_{2}$ claims a path of length $\ell-3$ with endpoints $\{x,y\}$ after $\ell-3$ moves, 
if $\{x,y\}\ne \{v_0,v_{\ell-3}\}$, we can assume $v_0\notin \{x,y\}$, then $P_1$ can pick a new vertex $v$ which does not appear in any claimed edges and then claim $v_{\ell-3}v$, otherwise suppose that $\{x,y\}= \{v_0,v_{\ell-3}\}$, then $P_1$ can do the same operation.
Observe that in both of the above cases, $P_2$ has no way to obtain a path of length $\ell-2$ with the same endpoints $\{v_{0},v\}$.

Therefore, $P_1$ can first claim a path of length $\ell-2$, denoted by $v_{0}v_{1}\cdots v_{\ell-2}$ after $\ell-2$ moves, so that $P_{2}$ cannot claim a path of length $\ell-2$ with the same endpoints $\{v_{0},v_{\ell-2}\}$. Without loss of generality, we can assume that $P_{2}$ does not claim any path of length $\ell-2$ with $v_{\ell-2}$ being one endpoint. $P_1$ can pick a new vertex $x$ which does not appear in any claimed edges and then claim $v_{0}x$ in his $(\ell-1)$-th round, and $P_{2}$ has to claim $v_{\ell-2}x$. Furthermore, $P_1$ can pick another new vertex $y$ which does not appear in any claimed edges and then claim $v_{0}y$ in his $\ell$-th round, and $P_{2}$ has to claim $v_{\ell-2}y$.  As $P_{2}$ does not claim a path of length $\ell-2$ with $v_{\ell-2}$ being one endpoint after $\ell-2$ moves, we can see $P_{2}$ does not claim any path of length $\ell-1$ after claiming $v_{\ell-2}x$ and $v_{\ell-2}y$. 

Finally, $P_{1}$ can pick a new vertex $z$ which does not appear in any claimed edges and then claim $v_{\ell-3}z$ in his $(\ell+1)$-th round, and $P_{2}$ has to claim both of $xz$ and $yz$, which is impossible. That means $P_{1}$ can win the game using $\ell+2$ moves.

\subsection{The game $R(K_{n}\sqcup K_{n}, K_{2,3}$)}
For a given positive integer \( i \), let \( e_{i}^{(1)} \) (resp. \( e_{i}^{(2)} \)) denote the edge claimed by \( P_1 \) (resp. \( P_2 \)) on the \( i \)-th move. Player \( P_1 \)'s initial objective is to claim a 4-cycle with favorable properties. As demonstrated in the following propositions, if \( P_1 \) succeeds in achieving this, they can secure a swift victory. This strategy serves as the foundation for \( P_1 \)'s overall approach to the game.

\begin{claim}\label{claim:C4}
Assume we are now in the $i$-th step where $i\in\mathbb{N}$, $P_{1}$ can win the game using at most $6$ more moves if the following holds.
$P_{1}$ has claimed a copy of $C_{4}:=abcd$ such that
\begin{enumerate}
    \item[\textup{(1)}] There is some edge $xy\in\{ab,bc,cd,ad\}$ such that $P_{2}$ has not claimed a path of length $2$ between $x$ and $y$.
    \item[\textup{(2)}] Furthermore, neither \( x \) nor \( y \) is contained in any copy of \( C_4 \) claimed by \( P_2 \).
    \item[\textup{(3)}] $P_{1}$ is not in check.
\end{enumerate}

\end{claim}
\begin{poc}
Without loss of generality, we assume that $xy=ab$. From rounds $i+1$ to $i+3$, $P_1$ could claim $cy_i$ for $i=[3]$ consecutively, for three unused vertices $y_i$. By property (3), $P_2$ must claim $ay_i$ for $i=[3]$ to prevent $P_1$ from claiming a $K_{2,3}$. From rounds $i+4$ to $i+5$, $P_1$ claims $dy_i$ for $i=[2]$ consecutively. In response, $P_2$ must claim $by_i$ for $i=[2]$ to prevent $P_1$ from claiming a copy of $K_{2,3}$. However, in round $i+6$, $P_1$ can claim a copy of $K_{2,3}$, namely $\{cy_1,cy_2,cy_3, dy_1,dy_2, dy_3\}$. Note that in this process, all of the vertices $y_i$ are new, which yields that after the $(i+4)$-th move, $e^{(2)}{(j)}$ cannot be contained in any copy of $K_{2,2}(1)$ claimed by $P_2$ for $i+1\leq j\leq i+4$. As a result, to claim a copy of $K_{2,3}$, $P_2$ needs to claim at least $2$ more edges. Therefore, $P_2$ can not prevent $P_i$'s victory in round $i+6$. 
\end{poc}

\begin{claim}\label{claim:C4+}
Assume we are now in the $i$-th step where $i\in\mathbb{N}$, $P_{1}$ can win the game using at most $5$ more moves if the following holds:
$P_{1}$ has claimed a copy of $K_{2,2}(1)$ which consists of a $C_4:=axby,$ and an edge $ac$ such that
\begin{enumerate}
    \item[\textup{(1)}] Neither \( x \) nor \( y \) is contained in any copy of \( C_4 \) claimed by \( P_2 \).
    \item[\textup{(2)}] $P_{1}$ is not in check.
\end{enumerate}
Moreover, if $P_2$ is $C_4$-free, then conditions (1) and (2) naturally hold.
\end{claim}
\begin{proof}
Without loss of generality, we assume that $x$ is not contained in any copy of $C_4$ claimed by $P_2$. Form rounds $i+1$ to $i+3$, $P_1$ could claim $yz_i$ for $i\in [3]$ consecutively, for three unused vertices $z_i$. Since $P_1$ is not in check and $x$ is not contained in any copy of $C_4$, we know that $P_2$ must claim $xz_i$ for $\in [3]$. In round $i+4$ $P_1$ claims $cz_1$. Clearly, $P_1$ is still not in check, so $P_1$ will get a $K_{2,3}$ with $c,y$ as the core in the next round by claiming one of $cz_2,cz_3$.
\end{proof}
\begin{figure}[htbp]
    \centering
    \begin{subfigure}[b]{0.4\textwidth}
    \centering 
    \begin{tikzpicture}[scale=0.3]
        \node (A) at (0,4) {$a$};
        \node (B) at (0,-4) {$b$};
        \node (C) at (4,4) {$c$};
        \node (D) at (4,-4) {$d$};
        \node (Y1) at (8,0) {$y_1$};
        \node (Y2) at (12,0) {$y_2$};
        \node (Y3) at (16,0) {$y_3$};
        \foreach \from in {A,B}
        \foreach \to in {Y1,Y2}
        \draw [red, thick] (\from) -- (\to);
        \draw [red, thick] (A)--(Y3);

        \foreach \from in {C,D}
        \foreach \to in {Y1,Y2,Y3}
        \draw [black, thick] (\from) -- (\to);
        \draw [black, thick](A) -- (B)--(C)--(D)--(A);
    \end{tikzpicture}
    \caption{The graph for Claim 3.1}
    \label{fig:KeyForP2}
    \end{subfigure}
    \begin{subfigure}[b]{0.4\textwidth}
    \centering 
    \begin{tikzpicture}[scale=0.3]
        \node (A) at (2,4) {$a$};
        \node (B) at (2,-4) {$b$};
        \node (X) at (0,0) {$x$};
        \node (Y) at (4,0) {$y$};
        \node (C) at (12,4) {$c$};
        \node (Z1) at (8,-4) {$z_1$};
        \node (Z2) at (12,-4) {$z_2$};
        \node (Z3) at (16,-4) {$z_3$};
        \foreach \from in {A,B}
        \foreach \to in {X,Y}
        \draw [black, thick] (\from) -- (\to);
        \draw [black, thick] (A) -- (C);
        \draw [black, thick] (Z2)--(Y) -- (Z1);
        \draw [black, thick] (Z3)--(Y);
        \draw [red, thick] (Z2)--(X) -- (Z1);
        \draw [red, thick] (Z3)--(X);
        \draw [black, thick] (C) -- (Z1);
    \end{tikzpicture}
    \caption{The graph for Claim 3.2}
    \label{fig:KeyForP2}
    \end{subfigure}
\end{figure}    

 Based on~\cref{claim:C4} and~\cref{claim:C4+}, we then describe the early strategy of $P_{1}$ as following picture. $P_{1}$ first picks an arbitrary edge, then $P_{1}$ checks whether $e^{(2)}(1)$ and $e^{(1)}(1)$ share a common vertex. 
 Note that in some cases, certain choices are isomorphic. We list only one representative in such cases. For example, in the second round of choices for $P_2$ in the first case, if $cx_1$ is selected, choosing $cx_2$, $y_1x_1$ or $y_1x_2$ would result in isomorphic graphs.
 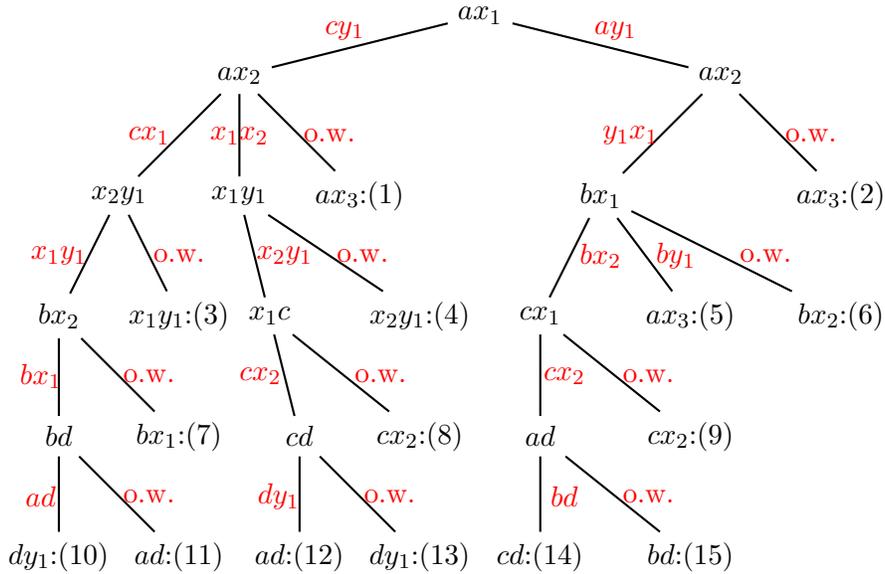
\begin{figure}[htbp]
    \centering 
    \begin{tikzpicture}[scale=0.4]
        \node (A1) at (0,-2) {$ax_1$};
        \node (B1) at (-8,-4) {$ax_2$};
        \node (B2) at (8,-4) {$ax_2$};
        \node (C1) at (-12,-8) {$x_2y_1$};
        \node (C2) at (-8,-8) {$x_1y_1$};
        \node (C3) at (-4,-8) {$ax_3$:(1)};
        \node (C4) at (4,-8) {$bx_1$};
        \node (C5) at (12,-8) {$ax_3$:(2)};

        \node(D1) at (-14,-12) {$bx_2$};
        \node(D2) at (-10,-12) {$x_1y_1$:(3)};
        \node(D3) at (-7,-12) {$x_1c$};
        \node(D4) at (-2,-12) {$x_2y_1$:(4)};
        \node(D5) at (2,-12) {$cx_1$};
        \node(D6) at (7,-12) {$ax_3$:(5)};
        \node(D7) at (12,-12) {$bx_2$:(6)};

        \node(E1) at (-14,-16) {$bd$};
        \node(E2) at (-10,-16) {$bx_1$:(7)};
        \node(E3) at (-6,-16) {$cd$};
        \node(E4) at (-2,-16) {$cx_2$:(8)};
        \node(E5) at (2,-16) {$ad$};
        \node(E6) at (7,-16) {$cx_2$:(9)};
 
        \node(F1) at (-14,-20) {$dy_1$:(10)};
        \node(F2) at (-10,-20) {$ad$:(11)};
        \node(F3) at (-6,-20) {$ad$:(12)};
        \node(F4) at (-2,-20) {$dy_1$:(13)};
        \node(F5) at (2,-20) {$cd$:(14)};
        \node(F6) at (7,-20) {$bd$:(15)};
        
    \draw [black, thick] (A1) -- (B1);
    \node at (-4.5,-2.5) {\textcolor{red}{$cy_1$}};
    \draw [black, thick] (A1) -- (B2);
    \node at (4.5,-2.5) {\textcolor{red}{$ay_1$}};

    \draw [black, thick] (B1) -- (C1);
    \node at (-11,-6) {\textcolor{red}{$cx_1$}};

    \draw [black, thick] (B1) -- (C2);
    \node at (-8,-6) {\textcolor{red}{$x_1x_2$}};

    \draw [black, thick] (B1) -- (C3);
    \node at (-5,-6) {\textcolor{red}{o.w.}};

    \draw [black, thick] (B2) -- (C4);
    \node at (5,-6) {\textcolor{red}{$y_1x_1$}};

    \draw [black, thick] (B2) -- (C5);
    \node at (11,-6) {\textcolor{red}{o.w.}};

    \draw [black, thick] (C1) -- (D1);
    \node at (-14,-10) {\textcolor{red}{$x_1y_1$}};

    \draw [black, thick] (C1) -- (D2);
    \node at (-10,-10) {\textcolor{red}{o.w.}};

    \draw [black, thick] (C2) -- (D4);
    \node at (-6.5,-10) {\textcolor{red}{$x_2y_1$}};

    \draw [black, thick] (C2) -- (D3);
    \node at (-3.9,-10) {\textcolor{red}{o.w.}};
    
    \draw [black, thick] (C4) -- (D5);
    \node at (4,-10) {\textcolor{red}{$bx_2$}};
    \draw [black, thick] (C4) -- (D6);
    \node at (6.5,-10) {\textcolor{red}{$by_1$}};
    \draw [black, thick] (C4) -- (D7);
    \node at (9.5,-10) {\textcolor{red}{o.w.}};

    \draw [black, thick] (D1) -- (E1);
    \node at (-14.6,-14) {\textcolor{red}{$bx_1$}};

    \draw [black, thick] (D1) -- (E2);
    \node at (-11,-14) {\textcolor{red}{o.w.}};

    \draw [black, thick] (D3) -- (E3);
    \node at (-7.3,-14) {\textcolor{red}{$cx_2$}};

    \draw [black, thick] (D3) -- (E4);
    \node at (-3.3,-14) {\textcolor{red}{o.w.}};

    \draw [black, thick] (D5) -- (E5);
    \node at (2.8,-14) {\textcolor{red}{$cx_2$}};

    \draw [black, thick] (D5) -- (E6);
    \node at (5.6,-14) {\textcolor{red}{o.w.}};

    \draw [black, thick] (E1) -- (F1);
    \node at (-14.6,-18) {\textcolor{red}{$ad$}};

    \draw [black, thick] (E1) -- (F2);
    \node at (-11,-18) {\textcolor{red}{o.w.}};

    \draw [black, thick] (E3) -- (F3);
    \node at (-6.7,-18) {\textcolor{red}{$dy_1$}};

    \draw [black, thick] (E3) -- (F4);
    \node at (-3,-18) {\textcolor{red}{o.w.}};

    \draw [black, thick] (E5) -- (F5);
    \node at (2.8,-18) {\textcolor{red}{$bd$}};

    \draw [black, thick] (E5) -- (F6);
    \node at (5.6,-18) {\textcolor{red}{o.w.}};
    
    \end{tikzpicture}
    \caption{The early strategy of $P_1$}
    \label{fig:KeyForK23}
 \end{figure}
 
In the following proof, we will frequently invoke Claim~\ref{claim:C4} or~\cref{claim:C4+} to demonstrate that $P_{1}$ can win the game within a finite number of moves. For convenience, we group some similar cases together for explanation.

\begin{figure}[htbp]
    \centering
\begin{subfigure}[b]{0.3\textwidth}
    \centering 
    \begin{tikzpicture}[scale=0.3]
        \node (A) at (0,0) {$a$};
        \node (X1) at (-4,-4) {$x_1$};
        \node (X2) at (0,-4) {$x_2$};
        \node (X3) at (4,-4) {$x_3$};
        \node (C) at (8,-4) {$c$};
        \node (Y1) at (8,0) {$y_1$};
        \node (B) at (6,-8) {$b$};
    \draw [black, thick] (A) -- (X1);
    \draw [black, thick] (A) -- (X2);
    \draw [black, thick] (A) -- (X3);
    \draw [black, dashed] (X2) -- (B) -- (X1);
    
    \draw [red, dashed] (X3) -- (B);
    \draw [red, thick] (C) -- (Y1);
    \end{tikzpicture}
    \caption{Graph:(1)}
    \label{fig:(1)}
\end{subfigure}
\begin{subfigure}[b]{0.3\textwidth}
    \centering 
    \begin{tikzpicture}[scale=0.3]
        \node (A) at (0,0) {$a$};
        \node (X1) at (-4,-4) {$x_1$};
        \node (X2) at (0,-4) {$x_2$};
        \node (X3) at (4,-4) {$x_3$};
        \node (Y1) at (6,0) {$y_1$};
        \node (B) at (6,-8) {$b$};
    \draw [black, thick] (A) -- (X1);
    \draw [black, thick] (A) -- (X2);
    \draw [black, thick] (A) -- (X3);
    \draw [black, dashed] (X2) -- (B) -- (X1);
    
    \draw [red, dashed] (X3) -- (B);
    \draw [red, thick] (A) -- (Y1);   
    \end{tikzpicture}
    \caption{Graph:(2)}
    \label{fig:(2)}
\end{subfigure}
\begin{subfigure}[b]{0.3\textwidth}
    \centering 
    \begin{tikzpicture}[scale=0.3]
        \node (A) at (0,0) {$a$};
        \node (B) at (6,-8) {$b$};
        \node (X1) at (-4,-4) {$x_1$};
        \node (X2) at (0,-4) {$x_2$};
        \node (X3) at (4,-4) {$x_3$};
        \node (Y1) at (6,0) {$y_1$};
    \draw [black, thick] (A) -- (X1);
    \draw [black, thick] (A) -- (X2);
    \draw [black, thick] (A) -- (X3);
    \draw [black, thick] (B) -- (X1);
    \draw [black, dashed] (X2) -- (B);
    
    \draw [red, dashed] (X3) -- (B);
    \draw [red, thick] (A) -- (Y1); 
    \draw [red, thick] (X1) -- (Y1);
    \draw [red, thick] (B) -- (Y1);
    \end{tikzpicture}
    \caption{Graph:(5)}
    \label{fig:(5)}
\end{subfigure}
\caption{The graphs (1) (2) and (5)}
\label{fig:case1}
\end{figure}
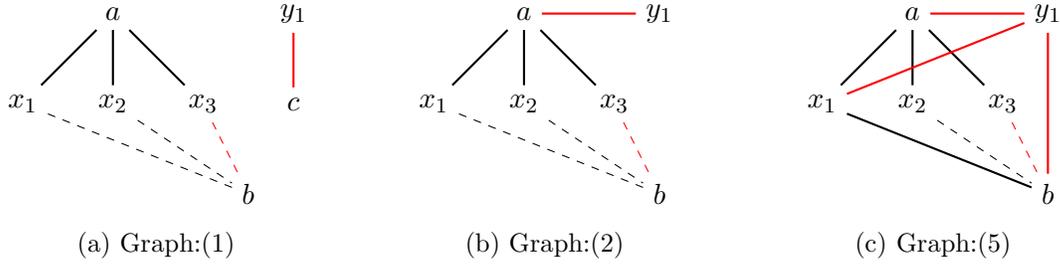 

\noindent \textbf{Case 1.} Proofs of (1), (2) and (5).
\begin{proof}
    For graphs (1) and (2), $P_1$ claims $bx_j$ for some $j\in [3]$ in round $4$, where $x_{j}$ has the highest degree in the graph claimed by $P_{2}$ and $b$ is an unused vertex. Then $P_1 $ claims $bx_j$ for some $j\in [3]$ which has not been claimed in round $5$. Without loss of generality, assume that $e^{(1)}(4)=bx_{1}$ and $e^{(1)}(5)=bx_{2}$. 
For graph (5), $P_1$ claims $bx_2$ or $bx_3$ which has not been claimed in round $5$.

We claim that no matter how $P_2$ claims before round 5, we can apply \cref{claim:C4+} ($x_1$ or $x_2$).
Let $H$ be the graph obtained by $P_2$ after the round $5$.
It is sufficient to show that: $P_1$ is not in check before round $6$ and neither $x_1$ nor $x_2$ is contained in $C_4$ in
graph $H$.
First, suppose that both $x_1$ and $x_2$ are vertices belonging to some $C_4$ in $H$. Since $e(H)=5$, we know that $x_1,x_2$ must belong to the same $4$-cycle, say $C$.
\begin{itemize}
    \item For graph (1), since $bx_1,bx_2 \notin E(H)$, $C$ does not contain $b$, which implies that $V(C) = \{x_1,x_2,c,y_1\}$. We derive that $P_2$ claims one of $\{x_1x_2, x_1c,x_1y_1,x_2c,x_2y_2\}$ in round $2$, a contradiction.
    \item For graph (2), $C$ does not contain $a$ and $b$ as $ax_1,ax_2,bx_1bx_2$ have been claimed by $P_1$. Then $e(H) \ge |E(C) \cup\{ay_1,bx_3\}| \ge 6$, a contradiction. 
    \item For graph (5), clearly, there is no $C_4$ containing both $x_1$ and $x_2$.
\end{itemize}
Thus, one of $x_1,x_2$ does not belong to the $C_4$ in $H$.
Now, suppose that $P_1$ is in check, which implies that $H$ is isomorphic to $K_{2,2}(1)$.
\begin{itemize}
    \item For graphs $(1)$ and $(2)$, note that $x_3$ is unused before round 3 and $b$ is unused before round 4. $P_2$ must claim an edge containing either $x_3$ or $b$ from round $3$ to $5$, since the sum of the degree of $x_3,b$ in $H$ is at least $4$. If $H$ does not contain $x_1$ and $x_2$, then the degree of $x_1,x_2$ is zero in $H$, which contradicts the choice of $bx_3$. Without loss of generality, we assume that $H$ contains $x_2$, which implies that $P_2$ claims $x_2y_1$ (it could also be $x_2c$ for graph (1)) in round $2$, a contradiction.

    \item For graph $(5)$, since $x_3 a$ has been claimed by $P_1$, $P_1$ can not be in check.
\end{itemize}
\end{proof}

\begin{figure}[htbp]
    \centering
\begin{subfigure}[b]{0.3\textwidth}
    \centering 
    \begin{tikzpicture}[scale=0.3]
        \node (A) at (0,0) {$a$};
        \node (X1) at (4,4) {$x_1$};
        \node (X2) at (4,-4) {$x_2$};
        \node (C) at (8,4) {$c$};
        \node (Y1) at (8,0) {$y_1$};
    \draw [black, thick] (A) -- (X1) -- (Y1) -- (X2) -- (A);

    \draw [red, thick] (C) -- (Y1);
    \draw [red, thick] (C) -- (X1);
    \end{tikzpicture}
    \caption{Graph:(3)}
    \label{fig:(3)}
\end{subfigure}
\begin{subfigure}[b]{0.3\textwidth}
    \centering 
    \begin{tikzpicture}[scale=0.3]
        \node (A) at (0,0) {$a$};
        \node (X1) at (4,4) {$x_1$};
        \node (X2) at (4,-4) {$x_2$};
        \node (C) at (8,4) {$c$};
        \node (Y1) at (8,0) {$y_1$};
    \draw [black, thick] (A) -- (X1) -- (Y1) -- (X2) -- (A);

    \draw [red, thick] (C) -- (Y1);
    \draw [red, thick] (X2) -- (X1);
    \end{tikzpicture}
    \caption{Graph:(4)}
    \label{fig:(4)}
\end{subfigure}
\begin{subfigure}[b]{0.3\textwidth}
    \centering 
    \begin{tikzpicture}[scale=0.3]
        \node (A) at (0,0) {$a$};
        \node (X1) at (4,4) {$x_1$};
        \node (X2) at (4,-4) {$x_2$};
        \node (B) at (8,0) {$b$};
        \node (Y1) at (0,4) {$y_1$};
    \draw [black, thick] (A) -- (X1) -- (B) -- (X2) -- (A);

    \draw [red, thick] (A) -- (Y1);
    \draw [red, thick] (Y1) -- (X1);
    \end{tikzpicture}
    \caption{Graph:(6)}
    \label{fig:(5)}
\end{subfigure}
\caption{The graphs (3) (4) and (6)}
\label{fig:case1}
\end{figure}
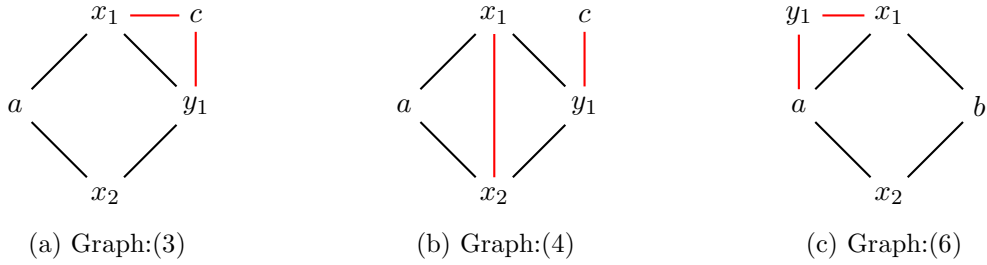

\noindent \textbf{Case 2.} Proofs of (3), (4) and (6).
\begin{proof}
    For graphs (3), (4) and (6), $P_2$ can claim two more edges. Clearly, no matter how $P_2$ claims these two edges, $P_2$ cannot obtain a $C_4$ containing $x_2$.
For graphs (3) and (4), unless $P_2$ claims $ca$ and $cx_2$, we can apply~\cref{claim:C4} ($x_2a$ or $x_2y_1$) to show that $P_1$ can win within 6 rounds.
Now we suppose that $P_2$ claims $ca$ and $cx_2$. 
\begin{itemize}

    \item For graph ($3$), $P_1$ will claim $ab$ for some unused vertex $b$ in the next round. Then no matter how $P_2$ claims the edge in the next round, $P_2$ is $C_4$-free. By \cref{claim:C4+}, we derive $P_1$ can win within 5 rounds. 
    \item For graph ($4$), we can apply~\cref{claim:C4} to show that $P_1$ can win within 5 rounds.
    \item For graph (6), since (3) and (6) are isomorphic, $P_1$ can also win within 6 rounds.
\end{itemize}
\end{proof}

\begin{figure}[htbp]
    \centering
\begin{subfigure}[b]{0.3\textwidth}
    \centering 
    \begin{tikzpicture}[scale=0.3]
        \node (A) at (0,0) {$a$};
        \node (X1) at (4,4) {$x_1$};
        \node (X2) at (4,-4) {$x_2$};
        \node (C) at (8,4) {$c$};
        \node (B) at (8,0) {$b$};
        \node (Y1) at (0,8) {$y_1$};
    \draw [black, thick] (A) -- (X1);
    \draw [black, thick] (B) -- (X2);
    \draw [black, thick] (B) -- (X1);
    \draw [black, thick](Y1) -- (X2) -- (A);

    \draw [red, thick] (C) -- (Y1);
    \draw [red, thick] (C) -- (X1);
    \draw [red, thick] (Y1) -- (X1);
    \end{tikzpicture}
    \caption{Graph:(7)}
    \label{fig:(3)}
\end{subfigure}
\begin{subfigure}[b]{0.3\textwidth}
    \centering 
    \begin{tikzpicture}[scale=0.3]
        \node (A) at (0,0) {$a$};
        \node (X1) at (4,-4) {$x_1$};
        \node (X2) at (4,4) {$x_2$};
        \node (Y1) at (8,4) {$y_1$};
        \node (C) at (8,0) {$c$};
    \draw [black, thick] (A) -- (X1) -- (C) -- (X2) -- (A);
    \draw [black, thick] (Y1) -- (X1);

    \draw [red, thick] (C) -- (Y1);
    \draw [red, thick] (X2) -- (X1);
    \draw [red, thick] (X2) -- (Y1);
    \end{tikzpicture}
    \caption{Graph:(8)}
    \label{fig:(8)}
\end{subfigure}
\begin{subfigure}[b]{0.3\textwidth}
    \centering 
    \begin{tikzpicture}[scale=0.3]
        \node (A) at (0,0) {$a$};
        \node (X1) at (4,4) {$x_1$};
        \node (X2) at (4,-4) {$x_2$};
        \node (B) at (8,4) {$b$};
        \node (C) at (8,0) {$c$};
        \node (Y1) at (0,4) {$y_1$};
    \draw [black, thick] (A) -- (X1) -- (C) -- (X2) -- (A);
    \draw [black, thick] (B) -- (X1); 

    \draw [red, thick] (A) -- (Y1);
    \draw [red, thick] (Y1) -- (X1);
    \draw [red, thick] (B) -- (X2);
    \end{tikzpicture}
    \caption{Graph:(9)}
    \label{fig:(9)}
\end{subfigure}
\caption{The graphs (7) (8) and (9)}
\label{fig:case1}
\end{figure}
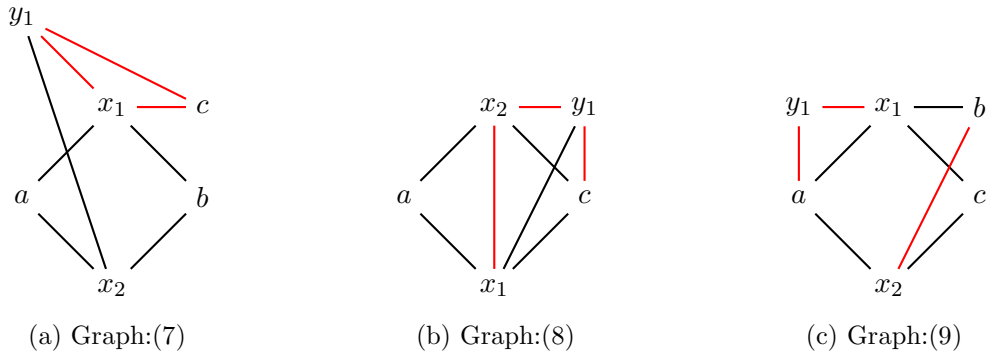

\noindent \textbf{Case 3.} Proofs of (7), (8) and (9).
\begin{proof}
  For graphs (7), (8) and (9), $P_2$ can claim two more edges. Clearly, no matter how $P_2$ claims those two edges, $P_2$ cannot obtain a $K_{2,2}(1)$, which implies that $P_1$ is not in check.
  \begin{itemize}
      \item For graphs (7) (8), no matter how $P_2$ claims those two edges, $P_2$ cannot obtain a $C_4$ containing vertex $a$.
      By \cref{claim:C4+}, $P_1$ can win within 6 rounds.
      \item For graph (9), no matter how $P_2$ claims those two edges, $P_2$ cannot obtain a $C_4$ containing $c$.
      By \cref{claim:C4+}, $P_1$ can win within 6 rounds.
  \end{itemize}

\end{proof}

\begin{figure}[htbp]
    \centering
\begin{subfigure}[b]{0.3\textwidth}
    \centering 
    \begin{tikzpicture}[scale=0.3]
        \node (A) at (0,0) {$a$};
        \node (X1) at (4,2) {$x_1$};
        \node (X2) at (4,-2) {$x_2$};
        \node (C) at (8,2) {$c$};
        \node (B) at (0,-4) {$b$};
        \node (Y1) at (8,-2) {$y_1$};
        \node (D) at (4,-6) {$d$};
    \draw [black, thick] (A) -- (X1);
    \draw [black, thick] (B) -- (X2);
    \draw [black, thick] (A) -- (X2);
    \draw [black, thick](Y1) -- (X2) -- (B)-- (D) -- (Y1);

    \draw [red, thick] (B) -- (X1);
    \draw [red, thick] (C) -- (Y1);
    \draw [red, thick] (C) -- (X1);
    \draw [red, thick] (Y1) -- (X1);
    \draw [red, thick] (A) -- (D);    
    \end{tikzpicture}
    \caption{Graph:(10)}
    \label{fig:(10)}
\end{subfigure}
\begin{subfigure}[b]{0.3\textwidth}
    \centering 
    \begin{tikzpicture}[scale=0.3]
        \node (A) at (0,0) {$a$};
        \node (X1) at (4,-2) {$x_1$};
        \node (X2) at (4,2) {$x_2$};
        \node (Y1) at (8,-2) {$y_1$};
        \node (C) at (0,-6) {$c$};
        \node (D) at (4,-6) {$d$};
    \draw [black, thick] (A) -- (X1) -- (C) -- (D) -- (A);
    \draw [black, thick] (Y1) -- (X1);
    \draw [black, thick] (A) -- (X2);

    \draw [red, thick] (C) -- (Y1);
    \draw [red, thick] (X2) -- (X1);
    \draw [red, thick] (X2) -- (Y1);
    \draw [red, thick] (D) -- (Y1);
    \draw [red, thick] (X2) -- (C);
    \end{tikzpicture}
    \caption{Graph:(12)}
    \label{fig:(12)}
\end{subfigure}
\begin{subfigure}[b]{0.3\textwidth}
    \centering 
    \begin{tikzpicture}[scale=0.3]
        \node (A) at (0,0) {$a$};
        \node (X1) at (4,4) {$x_1$};
        \node (X2) at (0,-4) {$x_2$};
        \node (B) at (8,4) {$b$};
        \node (C) at (8,0) {$c$};
        \node (Y1) at (0,4) {$y_1$};
        \node (D) at (4,-4) {$d$};
    \draw [black, thick] (A) -- (X1) -- (C) -- (D) -- (A);
    \draw [black, thick] (B) -- (X1); 
    \draw [black, thick] (A) -- (X2); 

    \draw [red, thick] (A) -- (Y1);
    \draw [red, thick] (Y1) -- (X1);
    \draw [red, thick] (B) -- (X2);
    \draw [red, thick] (C) -- (X2);
    \draw [red, thick] (B) -- (D);
    
    \end{tikzpicture}
    \caption{Graph:(14)}
    \label{fig:(14)}
\end{subfigure}
\caption{The graphs (10) (12) and (14)}
\label{fig:case1}
\end{figure}
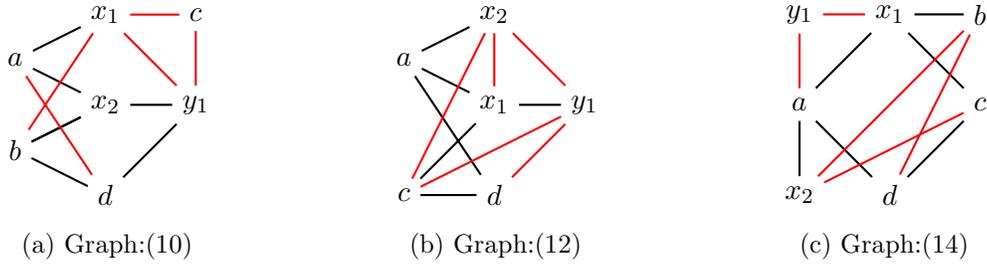

\noindent \textbf{Case 4.} Proofs of (10), (12) and (14).
\begin{proof}
    For graphs (10), (12) and (14), $P_2$ can claim one more edge.
    If $P_2$ is $C_4$-free, then by claim~\ref{claim:C4+}, $P_1$ can win within 5 rounds.
    Suppose that $P_2$ contains a $C_4$.
    \begin{itemize}
        \item  For graph (10), $P_2$ must claim $by_1$ or $bc$. Similar to the proof of \cref{claim:C4+}, $P_1$ could claim $x_2z_i$ for $i\in [3]$ consecutively, for three unused vertices $z_i$, then $P_2$ must claim $dz_i$ for $i\in [3]$. Thus, $P_1$ can obtain a $K_{2,3}$ by claiming two of the edges form $x_1y_i$ for $i\in [3]$.

        \item  For graph (12), $P_2$ must claim $dx_1$ or $dx_2$. By~ \cref{claim:C4+}, respectively, $P_1$ can win within 5 rounds.

        \item  For graph (14), it is impossible for $P_2$ to obtain a $C_4$ by adding just one more edge.
    \end{itemize}
\end{proof}
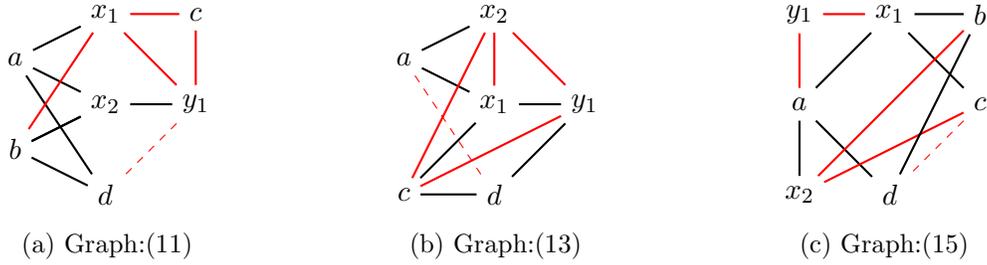
\begin{figure}[htbp]
    \centering
\begin{subfigure}[b]{0.3\textwidth}
    \centering 
    \begin{tikzpicture}[scale=0.3]
        \node (A) at (0,0) {$a$};
        \node (X1) at (4,2) {$x_1$};
        \node (X2) at (4,-2) {$x_2$};
        \node (C) at (8,2) {$c$};
        \node (B) at (0,-4) {$b$};
        \node (Y1) at (8,-2) {$y_1$};
        \node (D) at (4,-6) {$d$};
    \draw [black, thick] (A) -- (X1);
    \draw [black, thick] (B) -- (X2);
    \draw [black, thick] (A) -- (X2);
    \draw [black, thick](Y1) -- (X2) -- (B)-- (D);
    \draw [black, thick] (A) -- (D);  
    
    \draw [red, thick] (B) -- (X1);
    \draw [red, thick] (C) -- (Y1);
    \draw [red, thick] (C) -- (X1);
    \draw [red, thick] (Y1) -- (X1); 
    \draw [red, dashed] (Y1) -- (D);
    \end{tikzpicture}
    \caption{Graph:(11)}
    \label{fig:(10)}
\end{subfigure}
\begin{subfigure}[b]{0.3\textwidth}
    \centering 
    \begin{tikzpicture}[scale=0.3]
        \node (A) at (0,0) {$a$};
        \node (X1) at (4,-2) {$x_1$};
        \node (X2) at (4,2) {$x_2$};
        \node (Y1) at (8,-2) {$y_1$};
        \node (C) at (0,-6) {$c$};
        \node (D) at (4,-6) {$d$};
    \draw [black, thick] (A) -- (X1) -- (C) -- (D);
    \draw [black, thick] (Y1) -- (X1);
    \draw [black, thick] (A) -- (X2);

    \draw [red, thick] (C) -- (Y1);
    \draw [red, thick] (X2) -- (X1);
    \draw [red, thick] (X2) -- (Y1);
    \draw [black, thick] (D) -- (Y1);
    \draw [red, thick] (X2) -- (C);
    \draw [red, dashed] (A) -- (D);
    \end{tikzpicture}
    \caption{Graph:(13)}
    \label{fig:(13)}
\end{subfigure}
\begin{subfigure}[b]{0.3\textwidth}
    \centering 
    \begin{tikzpicture}[scale=0.3]
        \node (A) at (0,0) {$a$};
        \node (X1) at (4,4) {$x_1$};
        \node (X2) at (0,-4) {$x_2$};
        \node (B) at (8,4) {$b$};
        \node (C) at (8,0) {$c$};
        \node (Y1) at (0,4) {$y_1$};
        \node (D) at (4,-4) {$d$};
    \draw [black, thick] (D) --(A) -- (X1) -- (C);
    \draw [black, thick] (B) -- (X1); 
    \draw [black, thick] (A) -- (X2); 

    \draw [red, thick] (A) -- (Y1);
    \draw [red, thick] (Y1) -- (X1);
    \draw [red, thick] (B) -- (X2);
    \draw [red, thick] (C) -- (X2);
    \draw [black, thick] (B) -- (D);
    \draw [red, dashed] (C) -- (D);
    
    \end{tikzpicture}
    \caption{Graph:(15)}
    \label{fig:(15)}
\end{subfigure}
\caption{The graphs (11) (13) and (15)}
\label{fig:case1}
\end{figure}

\noindent \textbf{Case 5.} Proofs of (11), (13) and (15).
\begin{proof}
    For graphs (11), (13) and (15), $P_2$ can claim two more edges. Clearly, no matter how $P_2$ claims those two edges, $P_2$ cannot obtain a $K_{2,3}$. So $P_2$ must claim the dashed lines in the graphs.
    \begin{itemize}
    \item For graphs (13) and (15), it is impossible for $P_2$ to obtain a $C_4$ by adding just one more edge.
    By \cref{claim:C4+}, $P_1$ can win within 5 rounds.
    
    \item   For graph (11), it is impossible for $P_2$ to obtain a $C_4$ containing $x_2$ by adding just one more edge. By~\cref{claim:C4+}, $P_1$ can win within 5 rounds.
    \end{itemize}

\end{proof}

\section{Concluding remarks}
The main contribution of this paper is the discovery of a new family of graphs $K_{2,t+1}(t-2)$ and our demonstration that $P_{1}$ cannot win the game $R(K_{n}\sqcup K_{n},K_{2,t+1}(t-2))$ for $t\ge 3$. Moreover, the condition $t\ge 3$ is optimal, since $P_{1}$ can win the game when $t=2$. We suspect that these graphs could be promising candidates for answering Beck's major conjecture: there exists some graph $H$ such that player 1 cannot win the game $R(K_{n},H)$ within a bounded number of moves.

On the other hand, we are interested in identifying which graphs $H$ allow $P_{1}$ to win the game $R(K_{n},H)$. While we have already proven this for cycles, it would be interesting to explore additional graphs.

As mentioned earlier, the problem of significantly improving the upper bound \( L(K_n, K_5) \leq \left(\frac{3}{4} + o(1)\right) \binom{n}{2} \) remains open. More broadly, while it is conjectured that for any graph \( H \), \( L(K_n, H) \leq C_H \) for some constant \( C_H \), examples that substantially improve upon the trivial upper bound provided by the Turán number are quite rare. Hence, we believe the following problem is of independent interest.
\begin{problem}
    Find more graphs $H$ for which $L(K_{n},H)$ is sufficiently smaller than $ \textup{ex}(n,H)$. 
\end{problem}

For example, we say a graph $H$ is $2$-color-critical if there is some edge $e\in E(H)$ such that $H$ is not bipartite, but $H-e$ is bipartite. Next, we can establish an upper bound that is better than $(\frac{1}{4}+o(1))n^{2}$.

\begin{prop}\label{prop:bipartite_critical}
    For any given $2$-color-critical graph $H$, we have 
    \begin{equation*}
        L(K_{n},H)\le O(n^{2-\frac{1}{|H|}}).
    \end{equation*}
\end{prop}
\begin{proof}[Proof of~\cref{prop:bipartite_critical}]
Suppose that the game is not finished after \(C_H n^{2-\frac{1}{|H|}}\) moves, where \(C_H > 0\) is a sufficiently large constant. By the famous K\H{o}v\'{a}ri–S\'{o}s–Tur\'{a}n theorem~\cite{1954KST}, \(P_1\) must have claimed a copy of \(K_{|H|,|H|} = L \cup R\). Since the game has not ended, there must be some unclaimed edges in both parts \(L\) and \(R\). At this point, \(P_1\) can claim an arbitrary unclaimed edge, say \(e\). Note that \(H\) is a subgraph of \(K_{|H|,|H|} + e\), ensuring that \(P_1\) wins.
    
\end{proof}

\section*{Acknowledgement}
This work was initiated when Jun Gao and Zixiang Xu visited Nankai University in September 2023. The authors gratefully acknowledge the warm hospitality and support extended by several faculty members at Nankai University during their visit. The authors are grateful to Henri Ortm\"{u}ller for his valuable comments and suggestions.

\bibliographystyle{abbrv}
\bibliography{SRGame}

\end{document}